\documentclass[12pt]{amsart}

\gdef\private#1{}  \gdef\public#1{#1}  

\usepackage{amssymb}
\usepackage{xspace}
\usepackage{enumerate}
\usepackage{verbatim}
\usepackage{ifthen}
\usepackage[plain]{fancyref}
\usepackage{hyperref}
\private{\usepackage[color]{showkeys}}
\usepackage[normalem]{ulem}

\newtheorem{thm}{Theorem}
\newtheorem{lem}[thm]{Lemma}
\newtheorem{prop}[thm]{Proposition}
\newtheorem{cor}[thm]{Corollary}

\newtheorem{question}[thm]{Question}
\newtheorem{conj}[thm]{Conjecture}

\theoremstyle{definition}
\newtheorem{defn}[thm]{Definition}
\newtheorem*{notation}{Notation}

\theoremstyle{remark}

\newcommand*\fancyrefthmlabelprefix{thm}\frefformat{plain}{\fancyrefthmlabelprefix}{Theorem~#1}
\newcommand*\fancyreflemlabelprefix{lem}\frefformat{plain}{\fancyreflemlabelprefix}{Lemma~#1}
\newcommand*\fancyrefproplabelprefix{prop}\frefformat{plain}{\fancyrefproplabelprefix}{Proposition~#1}
\newcommand*\fancyrefcorlabelprefix{cor}\frefformat{plain}{\fancyrefcorlabelprefix}{Corollary~#1}
\newcommand*\fancyrefclaimlabelprefix{claim}\frefformat{plain}{\fancyrefclaimlabelprefix}{Claim~#1}
\newcommand*\fancyreffactlabelprefix{fact}\frefformat{plain}{\fancyreffactlabelprefix}{Fact~#1}
\newcommand*\fancyrefquestionlabelprefix{question}\frefformat{plain}{\fancyrefquestionlabelprefix}{Question~#1}
\newcommand*\fancyrefconjlabelprefix{conj}\frefformat{plain}{\fancyrefconjlabelprefix}{Conjecture~#1}
\newcommand*\fancyrefdefnlabelprefix{defn}\frefformat{plain}{\fancyrefdefnlabelprefix}{Definition~#1}
\newcommand*\fancyrefconstlabelprefix{const}\frefformat{plain}{\fancyrefconstlabelprefix}{Construction~#1}
\newcommand*\fancyrefsetuplabelprefix{setup}\frefformat{plain}{\fancyrefsetuplabelprefix}{Setup~#1}
\newcommand*\fancyrefexlabelprefix{ex}\frefformat{plain}{\fancyrefexlabelprefix}{Example~#1}
\newcommand*\fancyrefremlabelprefix{rem}\frefformat{plain}{\fancyrefremlabelprefix}{Remark~#1}
\frefformat{plain}{\fancyrefeqlabelprefix}{(#1)}
\newcommand*\fancyrefitemlabelprefix{item}\frefformat{plain}{\fancyrefitemlabelprefix}{(#1)}

\def\repeat#1#2 {\expandafter\gdef\csname B#1\endcsname {\mathbb{#1}}
  \ifthenelse{\equal{#2}{*}}{}{\repeat #2 }}
\repeat ABCDEFGHIJKLMNOPQRSTUVWXYZ*
\def\repeat#1#2 {\expandafter\gdef\csname C#1\endcsname {\mathcal{#1}}
  \ifthenelse{\equal{#2}{*}}{}{\repeat #2 }}
\repeat ABCDEFGHIJKLMNOPQRSTUVWXYZ*

\def\ug {{\underline{g}}}

\newcounter{last-index}
\newcommand{\xxx}[2][] {%
  \private{\ifthenelse{\equal{#1}{}}
    {\underline{$\bullet$}}{\uline{#1}}\marginpar{\tiny #2}}%
  \public{#1}\xspace}

\newcommand{\Fq} {{\mathbb{F}_q}} 

\newcommand{\Fclosed} {{\ensuremath{\overline{\,\mathbb{F}\,}}}\xspace}
\newcommand{\Fpclosed} {{\ensuremath{\Fclosed\kern-3pt_p}}\xspace}
\newcommand{\Fqclosed} {{\ensuremath{\Fclosed\kern-3pt_q}}\xspace}
\newcommand{\e} {{\ensuremath{\varepsilon}}\xspace}

\newcommand{\Frac}[2]{{\ensuremath{\textstyle\frac{#1}{#2}}}\xspace}
\newcommand{\gen} {{\rm gen}}

\newcommand{\cl}[1] {{\ensuremath{\overline{#1}}}\xspace}

\newcommand{\quotient}[2] {\ensuremath{\kern1.5pt{#1}\kern-2pt/\kern-2.5pt{#2}}}

\DeclareMathOperator{\diam}{diam}

\DeclareMathOperator{\degC}{deg_\BC}

\begin{document}
\title{Growth in linear groups
\private{\\\tiny\rm \today}}
\author{L\'aszl\'o Pyber
  and Endre Szab\'o
}
\public{
\date{\today}
\address{L\'aszl\'o Pyber and Endre Szab\'o\newline
A. R\'enyi Institute of Mathematics\newline
Hungarian Academy of Sciences\newline
P.O. Box 127\newline
H-1364 Budapest}
\email{pyber@renyi.hu}
\email{endre@renyi.hu}
\thanks{L.P. is supported in part by OTKA NK78439 and K84233}
\thanks{E.Sz. is supported in part by OTKA NK81203 and K84233}
}
\private{\fbox{\Huge\bf Private, do not distribute!}\\}

\maketitle

\begin{abstract}
  We give a description of non-growing subsets in linear groups,
  which extends the Product theorem for simple groups of Lie type.
  We also give an account of various related aspects of growth
  in linear groups.
\end{abstract}

\section{A Polynomial Inverse theorem in linear groups}
\label{sec:polyn-inverse-theor}

The following theorem was proved simultaneously and independently by
Breuillard, Green, Tao \cite{BrGrTao:ProductTheorem:2011}
and
Pyber, Szab\'o \cite{PyberSzabo:2010:growth}  in 2010.
\begin{thm} [Product theorem]
  \label{thm:Product-theorem}
  Let $L$ be a finite simple group of Lie
  type of rank $r$ and $A$ a generating set of $L$.  Then either
  $A^3=L$ or
  $$
  |A^3|\gg |A|^{1+\e}
  $$
  where $\e$ and the implied constant depend only on $r$.
\end{thm}

For $G= PSL(2,p)$, $p$ prime, this is a famous result of
Helfgott \cite{Helfgott:SL-2-p:2008}. 
For $PSL(3,p)$ resp. $PSL(2,q)$, $q$ a prime-power, this was proved
earlier by Helfgott \cite{Hefgott:SL-3-p:2011} resp.
Dinai \cite{Dinai:growth-SL2-finite-field:2011}
and Varj\'u \cite{varju:expansion-SLd-Z-square-free:2012}.

In \cite{PyberSzabo:2010:growth} we give various examples which show
that in the above result the
dependence of $\e$ on $r$ is necessary.
In particular we construct
generating sets of $SL(n,3)$ of size $2^{n-1}+4$ with $|A^3|< 100|A|$
for $n\ge3$.

For the groups $G=PSL(n,q)$
(which are simple groups of Lie type of rank $n-1$)
the Product theorem can be reformulated as follows:
If $A$ is a generating set of $G$, such that
$|A^3| < K|A|$ for some number $K\ge1$,
then $A$ is contained in $K^{c(n)}$ (i.e. polynomially many)
cosets of some normal subgroup $H$ of $G$ contained in $A^3$.
This (somewhat artificial) reformulation turns out to be quite useful
when we seek an extension of the Product theorem which describes
non-growing subsets of linear groups.

The Product theorem has quickly become a central result of finite
asymptotic group theory  with many applications. To describe a few of
these we introduce some graph-theoretic terminology. 
The \emph{diameter},
$\diam(X)$, of an undirected graph $X=(V,E)$ is the
largest distance between two of its vertices.
The \emph{girth} of $X$ is the length of the shortest cycle in $X$.
Given a subset
$A$ of the vertex set $V$ the \emph{expansion} of $A$, $c(A)$,
is defined to be the ratio $|\sigma(X)|/|X|$ where $\sigma(X)$ is the
set of vertices at distance $1$ from $A$.
A graph is a \emph{$C$-expander} for some $C>0$
if for all subsets $A$ of $V$
with $|A|< |V|/2$ we have $c(A)\ge C$.
A family of graphs is an \emph{expander family}
if all of its members are $C$-expanders
for some fixed positive constant $C$.

Let $G$ be a finite group and $S$ a symmetric (i.e. inverse-closed)
set of generators of $G$.  The Cayley graph $Cay(G,S)$ is a graph
whose vertices are the elements of $G$ and which has an edge from $x$
to $y$ if and only if $x=sy$ for some $s \in S$. Then the diameter of
$Cay(G,S)$ is the smallest number $d$
such that $S^d=G$.

The following classical 
conjecture is due to Babai
\cite{Babai-Seress:diameter-perm-group:1992}
\begin{conj}[Babai]
  \label{conj:babai}
  For every non-abelian finite simple group $L$ and every symmetric
  generating set $S$ of $L$ we have $\diam\big(Cay(L,S)\big) \le
  C\big(\log|L|\big)^c$ where $c$ and $C$ are absolute constants.
\end{conj}

The Product theorem easily implies the following.

\begin{cor}
 \fref{conj:babai} holds for simple groups of Lie type of bounded rank.
\end{cor}

The following conjecture
(see \cite{Liebeck-Nikolov-Shalev:Product-decomp-finite-simple-groups:2012})
is rather more recent.

\begin{conj} [Liebeck, Nikolov, Shalev]
  \label{conj:set-product-decomposition}
There exists an absolute constant $c$ such that if $L$ is a
finite simple group and $A$ is a subset of $L$ of size at least two,
then $L$ is a product of $N$ conjugates of $A$
for some $N\leq c\log|L|/ \log |A|$.
\end{conj}

By a deep (and widely applied) theorem of Liebeck and Shalev
\cite{Liebeck-Shalev:Diameter-of-simple-groups:2001} the conjecture
holds when $S$ is a conjugacy class or, more generally, a normal
subset (that is a union of conjugacy classes).
  
\begin{thm}
 \fref{conj:set-product-decomposition} holds for simple groups of Lie
 type of bounded rank.
\end{thm}

This is proved in \cite{Gill-Pyber-Short-Szabó:Shalev-conj:2012}
using the Product theorem and an extra trick which handles small subsets $S$.

Breuillard, Green, Guralnick and Tao
\cite{Breuillard-Green-Guralnick-Tao:Expansion-simple-groups:2012}
have announced a deep result of similar flavour.

\begin{thm}\label{thm:random-expander-BGGT}
  Let $G$ be a simple group of Lie type of rank $r$.
  Then for two random elements $x$, $y$
  the Cayley graph $Cay\big(G,\{x,y\}\big)$ is an $\varepsilon(r)$
  expander with probability going to $1$ as $|G| \to \infty$,
  where $\varepsilon(r)$ depends only on $r$.
\end{thm}

As an open problem Lubotzky \cite{Lubotzky:expander-pure-applied:2012}
suggested deciding whether this holds for all nonabelian finite
simple groups with an absolute constant $\varepsilon$.

The proof of the above theorem is based on the Product theorem,
some results in
\cite{Breuillard-Green-Guralnick-Tao:strongly-dense-free-semisimple:2012}
and the proof scheme
(termed the Bourgain-Gamburd expansion machine in
\cite{Tao:blog-Bourgain-Gamburd-expansion-machine:2012})
first used in the proof of the following theorem.
 
\begin{thm} [Bourgain-Gamburd
  \cite{Bourgain-Gambourd:Uniform-expansion-SL2-p:2008}]
  \label{thm:Bourgain-Gamburd-SL2-p}
  For any $0<\delta\in\BR$ there exists
  $\varepsilon=\varepsilon(\delta)\in\BR$
  such that for every prime $p$,
  if $S$ is a symmetric set of generators of $G=SL(2,p)$
  such that girth of the Cayley graph $Cay\big(G;S\big)$
  is at least $\delta\log p$ then 
  $Cay\big(G;S\big)$
  is an $\varepsilon$-expander.
\end{thm}

It seems quite plausible that this result extends to all nonabelian
finite simple groups. In view of the girth bounds in
\cite{Gamburd-Hoory-Shahshahani-Shalev-Virag:Girth-random-Cayley:2009}
this would imply \fref{thm:random-expander-BGGT}.
At present however such an extension is not even known to hold
for the groups $SL(2,q)$, $q$ a prime-power.

For the above applications of the Product theorem it is essential that
the size of a generating set $A$ is bounded by a polynomial of the
tripling constant $K=|A^3|/|A|$ (unless $A$ is very large). For a
discussion of related issues by Tao see
\cite{Tao:blog-Bourgain-Gamburd-expansion-machine:2012}.
Guided by this insight, and remarks of
Helfgott
\cite[page 764]{Hefgott:SL-3-p:2011},
and Breuillard, Green, Tao
\cite[Remark 1.8]{Breuillard-Green-Tao:structure-of-approximate-groups:2012},
and various discussions on Tao's blog,
in 2012 we proved the following
Polynomial Inverse theorem \cite{PyberSzabo:2012:HL-sejtes-soluble}.

\begin{thm} [Pyber, Szab\'o]
  \label{thm:PSz-Polynomial-Inverse-thm}
  Let $S$ be a symmetric subset of $GL(n,\BF)$
  satisfying
  $|S^3|\le K|S|$ for some $K\ge1$,
  where $\BF$ is an arbitrary field.
  Then $S$ is contained in the union of polynomially many
  (more precisely $K^{c(n)}$)
  cosets of a finite-by-soluble subgroup $\Gamma$ normalised by $S$.

  Moreover, $\Gamma$ has a finite subgroup $P$ normalised by $S$
  such that $\Gamma/P$ is soluble,
  and $S^3$ contains a coset of $P$.
\end{thm}

The theorem extends and unifies several earlier results.
Most importantly (to us)
it contains the Product theorem (for symmetric sets)
as a special case. 
For subsets of $SL(2,p)$ and $SL(3,p)$
similar results were obtained by Helfgott
\cite{Helfgott:SL-2-p:2008, Hefgott:SL-3-p:2011}.

In characteristic $0$ the above theorem was first proved by
Breuillard, Green and Tao \cite{BrGrTao:ProductTheorem:2011}.
The earliest result in this direction for subsets of $SL(2,\BR)$
is due to Elekes and Kir\'aly
\cite{Elekes-Kiraly:Projective-mappings:2001}.
The proof in \cite{BrGrTao:ProductTheorem:2011}
uses the fact that 
a virtually soluble subgroup of $SL(n,\BC)$
has a soluble subgroup of $n$-bounded index,
which is no longer true in positive characteristic.

Finally the theorem implies a result of Hrushovski for linear groups
over arbitrary fields obtained by model-theoretic tools
\cite{Hrushovski:approximate-subgroups:2012}.
In his theorem the structure of $\Gamma$ is described  in a less
precise way and the number of covering cosets is only bounded by some
large function of $n$ and $K$. 

It seems possible that a result similar to
\fref{thm:PSz-Polynomial-Inverse-thm}
can be proved with $H$ nilpotent
(possibly at the cost of loosing the
normality of $H$ and $P$ in $\langle S\rangle$).
Indeed in characteristic $0$ such a result follows by combining
\cite{BrGrTao:ProductTheorem:2011}
with
\cite{Breuillard-Green:Approximate-groups-II-solvable-linear:2011},
and for prime fields it follows from 
\cite{PyberSzabo:2010:growth} and the results of 
\cite{Gill-Helfgott:solvable-growth:2010} for soluble
groups. In general this may be technically quite challenging even
though soluble linear groups have a nilpotent-by-abelian
subgroup of $n$-bounded index.

It may also be possible that a more general polynomial inverse theorem
holds (see \cite{Breuillard:survey:2012}).
We pose the following question which bypasses abelian and finite
obstacles.

\begin{question}
  Let $S$ be a finite symmetric subset of a group $G$ such that
  $|S^3| \le K|S|$ for some $K\ge1$. Is it true that S is contained in
  $K^{c(G)}$ cosets of some virtually soluble subgroup of $G$? 
\end{question}

The above question may be viewed as a counterpart of the Polynomial
Freiman-Ruzsa Conjecture which asserts that
(a variant of)
Freiman's famous Inverse
theorem holds with polynomial constants.
The existence of some huge bound $f(K)$ for the number of covering cosets,
as above, follows from the very general Inverse theorem
of Breuillard, Green and Tao
\cite{Breuillard-Green-Tao:structure-of-approximate-groups:2012}. 
See Breuillard's survey \cite{Breuillard:survey:2012}
for a detailed discussion of these issues.

\fref{thm:PSz-Polynomial-Inverse-thm} shows that the answer is
positive for the groups $G=SL(n,\BF)$. It would interesting to
investigate this question for various groups of intermediate word
growth, such as the Grigorchuk groups
(see e.g. \cite{de-la-Harpe:Topics-geometric-group-theory:2000}).

 It would be extremely interesting if the number of cosets required would
be bounded by $K^c$ for some absolute constant $c$ for all groups
$G$. Obtaining such a result for all linear groups $G$ (in which $c$ does
not depend on the dimension) already seems to require some essential
new ideas. 

\section{Applications of the Sarnak-Xue-Gowers trick}

A subset $X$ of a group $G$ is called product-free
if there are no solutions to $xy=z$ with $x,y,z\in X$.

Answering an 1985 question of Babai and S\'os
\cite{Babai-Sos:Sidon-Cayley:1985},
Gowers \cite{Gowers:Quasirandom-groups:2008}
showed that the group $G=PSL(2,p)$ has no product-free subset of size greater
than $c|G|^{\frac89}$ for some $c>0$.
The proof in \cite{Gowers:Quasirandom-groups:2008}
is closely related to an argument of Sarnak and Xue
\cite{Sarnak-Xue:Automorphis-repr-multiplicity:1991},
which also plays an important role in the proof of
\fref{thm:Bourgain-Gamburd-SL2-p}.

\begin{notation}
  For any group $G$ let $\degC(G)$ denote the minimum degree of a
  non-trivial complex representation.
\end{notation}

It was observed in
\cite{Nikolov-Pyber:Product-decomposition-quasirandom-Jordan:2011},
that as a byproduct of the results of Gowers
\cite{Gowers:Quasirandom-groups:2008}
one obtains the following extremely useful corollary.

\begin{cor} \label{cor:gowers-trick}
  Let $G$ be a group of order $n$ such that $\degC(G)=k$.
  If $A$ is a subset of $G$ such that
  $|A|>\frac{n}{k^{1/3}}$,
  then we have $A^3=G$.
\end{cor}

For groups of Lie type rather strong lower bounds on the minimal degree of a
complex representation are known
\cite{Landazuri-Seitz-Minimal-degree-proj-repr:1974}.
Combining these bounds with \fref{cor:gowers-trick}
we obtain e.g. the following.

\begin{prop} \label{prop:PSLn-gowers-trick}
  Let $A$ be a subset of $G=PSL(n,q)$
  of size at least $2|G|/q^{\frac{n-1}{3}}$.
  Then we have $A^3=G$.
\end{prop}

This of course implies the Product theorem for very large subsets.

As another interesting application,
in \cite{Nikolov-Pyber:Product-decomposition-quasirandom-Jordan:2011}
the following ``Waring-type'' theorem is proved.

\begin{thm} \label{thm:waring}
  Let $w$ be a nontrivial group word.
  Let $L$ be a finite simple group of Lie type of rank $r$
  over the field $\Fq$.
  Let $W$ be any subset of $w(L)$ such that
  $|W|\ge\big|w(L)\big|\big/q^{r/13}$.
  Then there exists a positive integer $N$
  depending only on $w$ such that if $|L|>N$ then
  $$
  W^3=L \;.
  $$
\end{thm}

Earlier, as the main result of a difficult paper
Shalev \cite{Shalev:Word-maps-noncommutative-Waring:2009}
has obtained the same result in the case $W=w(L)$
(allowing $L$  to be also an alternating group).

An advantage of the above sparse version is that
one can impose further restrictions on $W$.
For example one can require that no two elements of $W$
are inverses each other, or images of each other under
Frobenius automorphisms.
It would be interesting to obtain a sparse ``Waring-type''
theorem for the alternating groups.
 
The proof of \fref{thm:waring}
is relatively short compared to the one in
\cite{Shalev:Word-maps-noncommutative-Waring:2009}.
One only has to show that
for simple groups of Lie type
the sets $w(L)$ are large enough.

In a major recent work Larsen, Shalev and Tiep
\cite{Larsen-Shalev-Tiep:Waring-simple-groups:2011}
have shown that in fact,
for a nontrivial word $w$,
$w(L)^2=L$ holds for large enough finite simple groups.

In certain situations, in which \fref{cor:gowers-trick}
cannot be applied directly,
the following extension due to Babai, Nikolov and Pyber
\cite{Babai-Nikolov-Pyber:Product-growth-Mixing:2008}
has turned out to be quite useful.

\begin{lem}\label{lem:turbo-gowers-trick}
  Let $G$ be a group of order $n$ such that $\degC(G)=k$.
  Let $A_1,\dots,A_t$ ($t\ge3$)
  be nonempty subsets of $G$.
  If 
  $$
  \prod_{i=1}^t\big|A_i\big|\ge
  \frac{n^t}{k^{t-2}}
  $$
  then
  $
  \prod_{i=1}^tA_i = G
  $.
\end{lem}

A major result of Nikolov and Segal
\cite{Nikolov-Segal:profinite-I:2007}
\cite{Nikolov-Segal:profinite-II:2007} 
states that finite index subgroups of finitely generated profinite groups are
open, as conjectured by Serre.
Improving  their methods, Nikolov and Segal
\cite{Nikolov-Segal:powers-in-finite-groups:2011}
later proved the following theorem,
which also implies Serre's conjecture:

\begin{thm} \label{thm:Power-width}
  Let $d,q\in\BN$ and
  $G$ a finite $d$-generated group.
  Let $G^q$ be the subgroup of $G$ generated by the $q$-th powers.
  There exists a function $f(d,q)$
  such that
  every element of $G^q$ is a product of at most $f(d,q)$  powers $x^q$.
\end{thm}

Recently Nikolov and Segal
\cite{Nikolov-Segal:generators-commutators-in-finite-groups:2012}
obtained much shorter proofs as well as vast generalisations of their results.
A key tool in obtaining these is \fref{lem:turbo-gowers-trick}.

In \cite{Babai-Nikolov-Pyber:Product-growth-Mixing:2008}
\fref{lem:turbo-gowers-trick}
is used to derive the following improvement of a result from
\cite{Liebeck-Pyber:Finite-linear-groups-bounded-generation:2001}.

\begin{thm} \label{thm:5-sylow}
  Let $L$ be a finite simple group of Lie type in characteristic $p$.
  Then $L$ is a product of 5 Sylow $p$-subgroups.
\end{thm}

As an illustration of the proof consider the Ree groups
$L={}^2F_4(q)$
(a difficult case in
\cite{Liebeck-Pyber:Finite-linear-groups-bounded-generation:2001}). 
Take a pair of Sylow $p$-subgroups $U$ and $V$ in $L$
such that
$U\cap V=\{1\}$.
Using the tables in
\cite{Kleidman-Liebeck:subgroup-structure-classical:1990}
it follows, that we have
$$
\frac{|L|}{|UV|} <
\degC(L)^{\frac12} \;.
$$
Hence \fref{lem:turbo-gowers-trick}
with $t=4$ immediately gives that 
$$
L=(UV)(VU)(UV)(VU) = UVUVU \;.
$$

Combining \fref{thm:5-sylow}
with an argument from
\cite{Liebeck-Pyber:Finite-linear-groups-bounded-generation:2001}
one obtains a similar improvement of another result in
\cite{Liebeck-Pyber:Finite-linear-groups-bounded-generation:2001}.

\begin{thm}
  Let $\BF$ be a field of characteristic $p$
  and $G$ a finite subgroup of $GL(n,\BF)$
  generated by its Sylow $p$-subgroups.
  If $p$ is sufficiently large compared to $n$
  then $G$ is a product of 5 of its Sylow $p$-subgroups.
\end{thm}

This is a generalisation of the following curious result of
Hrushovski and Pillay 
\cite{Hrushovski-Pillay:Definable-subgroups:1995},
originally obtained by model-theoretic tools.

\begin{thm} \label{thm:Hrushovski-Pillay}
  Let $p$ be a prime,
  and $G$ a subgroup of $GL(n,p)$ generated by elements of order $p$.
  Then $G$ can be written as a product $G=C_1C_2\dots C_k$
  where the $C_i$ are cyclic subgroups of order $p$
  and $k$ depends only on $n$.
\end{thm}

Returning to our main theme next we present a (previously unpublished)
``baby version'' of the Product theorem.
This result was in fact the starting point of the authors'
joint work concerning growth in groups.
It is proved by a ``greedy argument'' based on \fref{cor:gowers-trick}.
The proof of the Product theorem may be viewed as an analogous greedy argument
based on certain Larsen-Pink type inequalities
(see \fref{sec:proof-prod-theor}).

First we quote two combinatorial results which also play an important role
in the proof of the Product theorem.
As noted in \cite{Hefgott:SL-3-p:2011},
the first one is essentially due to Ruzsa
(see \cite{Ruzsa:noncommutative-Plunnecke:2010}
\cite{Ruzsa-Turjanyi:additive-bases-of-integers:1985}).

\begin{prop} [\cite{Hefgott:SL-3-p:2011}]
  \label{prop:Helfgott-3-is-enough}
  Let $1\in S$ be a symmetric finite subset of a group
  and $k\ge3$ an integer. Then
    $$
    \frac{\big|S^k\big|}{|S|}\le
    \left(\frac{\big|S^3\big|}{|S|}\right)^{k-2}
    $$
\end{prop}

That is, if a small power of $S$ is much larger than $S$ itself,
then $S^3$ is already much larger than $S$.

The next lemma is a slight extension of result
from\cite{Hefgott:SL-3-p:2011}.

\begin{lem} [\cite{PyberSzabo:2010:growth}]
  \label{lem:induce-from-growing-subgroup}
  Let $1\in S$ be a symmetric finite subset of a group $G$,
  and $H$ a subgroup of $G$.
  Then for all integers $k>0$ one has
  $$
  \Frac{\big|S^k\cap H\big|}{\big|S^2\cap H\big|} \le
  \Frac{\big|S^{k+1}\big|}{\big|S\big|} \;.
  $$
\end{lem}

That is, growth in any subgroup $H$ implies growth in $G$ itself.

Here is the simplest version of our ``Baby Product theorem''.

\begin{thm}\label{thm:Baby-Product-thm}
  Let $S$ be a symmetric generating set of the group $G=SL(n,q)$
  (where $q\ge 4$ is a prime power) 
  and $H$ a subgroup of $G$ isomorphic to $SL(2,q)$.
  If $S$ contains $H$ then
  either $S^3=G$
  or
  $$
  \big|S^3\big|>|S|\cdot q^{\frac1{1000}} \;.
  $$
\end{thm}
\begin{proof}
  We have $\degC(H)\ge\frac{q-1}2$.
  Assume first, that $S^6$ does not contain all conjugates of $H$ in $G$.
  Since $S$ generates $G$,
  there exists a conjugate $H_0$ of $H$
  such that
  $H_0$ is contained in $S^6$, and an $S$-conjugate, say $H_0^s$
  of $H_0$ is not contained in $S^6$.
  Then we have
  $$
  \big|S^2\cap H_0^s\big| \le \big|H_0\big|\cdot \degC(H)^{-\frac13}
  $$
  since otherwise 
  $S^6$ would contain $H_0^s$
  by \fref{cor:gowers-trick}.
  It follows from the assumptions on $H_0$
  that $H_0^s$ is contained in $S^8$.
  Hence by \fref{lem:induce-from-growing-subgroup}
  we have
  $$
  \frac{|S^9|}{|S|} \ge
  \frac{|S^8\cap H_0^s|}{|S^2\cap H_0^s|} \ge
  \degC(H)^{\frac13} \;.
  $$
  By \fref{prop:Helfgott-3-is-enough}
  this implies
  $$
  \left(\frac{|S^3|}{|S|}\right)^7 \ge
  \frac{|S^9|}{|S|} \ge
  \degC(H)^{\frac13} \;.
  $$
  Therefore
  $$
  \big|S^3\big| \ge
  |S|\cdot \degC(H)^{\frac1{27}} =
  |S|\cdot \left(\frac{q-1}2\right)^{\frac1{27}}
  $$
  and our statement follows in this case.

  Assume now that $S^6$ contains all conjugates of $H$,
  but $S^3\neq G$.
  In this case $S^6$ contains a non-central conjugacy class $C$ of $G$
  together with $C^{-1}$.
  By a result of Lawther and Liebeck
  \cite{Lawther-Liebeck:diameter-Lie-type-conjugacy-class:1998}
  (see also \cite{Ellers-Gordeev-Herzog:Covering-numbers-Chevalley:1999}),
  if $\CK$ is any non-central conjugacy class of $SL(n,q)$
  ($q\ge4$),
  then $\big(\CK\cup\CK^{-1}\big)^{40n}=SL(n,q)$.

  This implies $S^{240n}=G$.
  On the other hand, since $S^3\neq G$,
  by \fref{prop:PSLn-gowers-trick}
  (which also holds for $SL(n,q)$),
  we have $|S|\le 2|G|\big/q^{\frac{n-1}3}$.
  By \fref{prop:Helfgott-3-is-enough}
  $$
  \left(\frac{|S^3|}{|S|}\right)^{240n-2}\ge
  \frac{|S^{240n}|}{|S|} =
  \frac{|G|}{|S|} \ge
  \frac12 q^{\frac{n-1}3}
  $$
  and our statement follows.
\end{proof}

Essentially the same result continues to hold when $G$ is an arbitrary simple
group of Lie type over $\Fq$.
More significantly,
if in \fref{thm:Baby-Product-thm}
we replace $H$ by any subgroup with $\degC(H)\ge q^{\CO(1)}$,
we still obtain the conclusion $\big|S^3\big|\ge |S|\cdot q^{\CO(1)}$
(unless $S^3=G$).
Here $H$ could be e.g. a Suzuki subgroup, or a subfield subgroup
$SL(n,q_0)$ for some $\BF_{q_0}\le\Fq$.

\begin{defn}
  A family of finite groups $G_i$ is an expanding family
  if, for some natural number $k$,
  there exist generating subsets $S_i$ of $G_i$ of size at most $k$
  such that the graphs $Cay(G_i,S_i)$ form an expander family.
\end{defn}

In \cite{Kassabov-Lubotzky-Nikolov:simple-group-expanders:2006}
Kassabov, Lubotzky and Nikolov announced
that all finite simple groups,
except possibly the Suzuki groups,
form an expanding family.
The same is now known for the Suzuki groups as well
by \cite{Breuillard-Green-Tao:suzuki-expander:2011},
and more generally by \fref{thm:random-expander-BGGT}.

As an important tool in the proof of this result
Lubotzky \cite{Lubotzky:Lie-type-expanders:2011}
has shown, that a simple group of Lie type of rank $r$
over the field $\Fq$
(not a Suzuki group)
decomposes as a product of $f(r)$ subgroups isomorphic to
$SL(2,q)$ or $PSL(2,q)$.
The argument in \cite{Lubotzky:Lie-type-expanders:2011}
is based on model theory.
\fref{thm:Baby-Product-thm},
extended to simple groups of Lie type,
easily implies such a result.
Our proof (assuming known results)
is short and elementary.

Another group-theoretic proof,
that gives a quadratic bound for $f(r)$,
which uses \fref{lem:turbo-gowers-trick} in a different way,
is given in \cite{Liebeck-Nikolov-Shalev:product-of-SL2:2011}.
That proof does not yield growth results,
and does not apply to, say, Suzuki subgroups

Somewhat surprisingly the following affine variant of
\fref{thm:Baby-Product-thm} is also true.

\begin{lem}\label{lem:conjugation-trick}
  Let $ H$ be a $d$-generated finite group with $\degC( H)>K^{21}$,
  and $A$ a $\BZ H$-module.
  Let $0\in S\subseteq A$ be
  a symmetric $ H$-invariant set generating $A$.
  Then (with multiplicative notation) \
  $|S^3|>K\cdot|S|$
  or $S^{7d}$ contains the submodule $[ H,A]$.
\end{lem}

This plays an essential role in proving
the Polynomial Inverse theorem
for linear groups over finite fields.

We end this section with a question vaguely related to super-strong
approximation (see e.g. \cite{Breuillard:survey:2012}).

\begin{question} \label{question:expanding-diameter}
  Let $G_i$ be an expanding family of finite groups.
  Is it true that
  for all $i$ and every symmetric generating set $S_i$ of $G_i$
  we have $\diam\big(Cay(G_i,S_i)\big) \le  C\big(\log|G_i|\big)^c$
  where the constants $c$ and $C$ depend only on the family?
\end{question}

Note that answering an earlier question of Lubotzky and Weiss 
\cite{Lubotzky-Weiss:groups-and-expanders:2009}
in the negative,
in \cite{Alon-Lubotzky-Wigderson:semidirect-zigzag-graphs:2001}
some expanding families of groups $G_i$ are constructed together with
bounded sized generating sets $S_i$ such that the corresponding Cayley graphs
do not form an expander family
(see also \cite{Meshulam-Wigderson:expanders-in-groupalgebras:2004}).
The alternating groups also have these properties,
but the fact that they form an expanding family was proved somewhat later
in the breakthrough paper of Kassabov
\cite{Kassabov:Symmetric-group-expanders:2007}.

Note also that
since nonabelian finite simple groups form an expanding family,
a positive answer to \fref{question:expanding-diameter}
would imply Babai's conjecture.

\section{On the proof of the Product theorem}
\label{sec:proof-prod-theor}

In this section we describe the main ideas in the proof of the Product theorem
for $SL(n,q)$.
Simple groups of Lie type can be handled by essentially the same argument.
A far-reaching generalisation of the argument plays an important role in the
proof of the Polynomial Inverse theorem.

In the course of the proof we obtain various results which say that if $L$ is
a ``nice'' subgroup of an algebraic group $G$,
generated by a set $A$, then $A$ grows in
some sense.
These were motivated by earlier results of Helfgott
\cite{Helfgott:SL-2-p:2008}, \cite{Hefgott:SL-3-p:2011},
Hrushovski-Pillay \cite{Hrushovski-Pillay:Definable-subgroups:1995}, and
\fref{thm:Baby-Product-thm}.

Assume for example that $A$ generates
$L=SL(n,q)$ ($q$ a power of a prime $p$),
which is a subgroup of $G=SL(n,\Fpclosed)$,
and ``$A$ does not grow''
i.e. $|A^3|$ is not much larger than $|A|$.  Using an ``escape from
subvarieties'' argument
Helfgott \cite{Hefgott:SL-3-p:2011} proved the following useful lemma:
If $T$ is a
maximal torus in $G$ then $|T \cap A|$ is not much larger than
$|A|^{1/(n+1)}$ .  This is natural to expect for dimensional reasons
since $\dim(T)/\dim(G)=(n-1)/(n^2-1)=1/(n+1)$. 

In particular we see that, if $A$ contains a maximal torus of $L$,
then $A$ is a very large subset of $L$,
hence $A^3=L$ by \fref{cor:gowers-trick}. 
This consequence of Helfgott's lemma
is very similar to our Baby Product theorem.

What is the right generalisation of these results?
Our first answer to this question \cite{Pyber-Szabo:honlapon},
which is subsumed by
\cite[Theorem 49]{PyberSzabo:2010:growth},
is the following inequality.

\begin{thm} \label{thm:LP-inequality}
  Let $\varepsilon>0$ be a fixed constant.
  Let $G$ be a linear algebraic group defined over $\Fq$
  which satisfies the following conditions.
  \begin{enumerate}[(a)]
  \item 
    The centraliser of $G(\Fq)$ (the group of $\Fq$-points)
    in $G$ is finite.
  \item 
    $G(\Fq)$
    does not normalise any closed subgroup $H< G$ with
    $0<\dim(H)<\dim(G)$.
  \end{enumerate}
  Let $1\in A$ be a generating set of $G(\Fq)$
  and $V$ a subvariety of $G$ of positive dimension.
  If $A\cap V$ is large enough
  (i.e. greater than some appropriate function of $\varepsilon$, $\dim(G)$,
  and the degrees of the varieties $G$ and $V$)
  then
  $$
  \big|A^m\big| \ge
  \big|A\cap V\big|^{(1-\varepsilon)\dim(G)/\dim(V)}
  $$
  for some $m$ which depends only on $\varepsilon$ and $\dim(G)$.
\end{thm}

Taking $G$ to be $SL(n,\Fpclosed)$ and $T$ a maximal torus in $G$
we obtain Helfgott's lemma.
Here we use the fact, essential to Helfgott's approach,
that if $A^3$ is not much larger than $A$
then $A^m$ is not much larger either
(see \fref{prop:Helfgott-3-is-enough}).

Breuillard, Green and Tao \cite{BrGrTao:ProductTheorem:2011}
arrived at a similar, somewhat stronger inequality
by taking hints from a paper of Hrushovski
\cite{Hrushovski:approximate-subgroups:2012}.
(Our inequality was obtained in March 2009,
independently of the work of Hrushovski
\cite{Hrushovski:approximate-subgroups:2012}).
For a formulation of their result called a Larsen-Pink type inequality
in \cite{BrGrTao:ProductTheorem:2011}
see \cite{Breuillard:survey:2012}.

Anyway, the inequality says that
if a generating set of $L=SL(n,q)$ does not grow
then it is distributed in a ``balanced way'' in $L$.

How do we break this balance?
We have to start conjugating!
More precisely, we have to answer another question.
Which are the ``interesting subvarieties'' to which \fref{thm:LP-inequality}
should be applied to?
Centralisers and conjugacy classes,
as explained below.

For an element $x$ consider the conjugation map
$g\to x^g$.
The image of this map is $cl(x)$, the conjugacy class of $x$ in $G$,
and the fibres are cosets of the centraliser $\CC_G(x)$.
For the closure $\cl{cl(x)}$ of the set $cl(x)$
this implies, that
$\dim\big(\CC_G(x)\big) + \dim\big(\cl{cl(x)}\big) = \dim(G)$.
It follows that an upper bound for $\big|A\cap\cl{cl(x)}\big|$
implies a lower bound for $\big|A\cap\CC_G(x)\big|$
which matches the upper bound given by \fref{thm:LP-inequality}.

Maximal torii are the centralisers of their regular semisimple elements
(that is elements with all eigenvalues distinct).

Let us say that $A$ \emph{covers} a maximal torus $T$ if
$\big|T \cap A\big|$
contains a regular semisimple element.  We obtain the following
fundamental dichotomy (see \cite[Lemma~60]{PyberSzabo:2010:growth}):

\emph{ Assume that a generating set A does not grow
  \begin{enumerate}[i)]
  \item If $A$ does not cover a maximal torus $T$ then $\big|T \cap
    A\big|$ is not much larger than $|A|^{1/(n+1) -1/(n^2-1)}$.
  \item If $A$ covers $T$ then $\big|T \cap AA^{-1}\big|$ is not much
    smaller than $|A|^{1/(n+1)}$.
  \end{enumerate}
}

To break the balance when $A$ is not very large,
we only have to start conjugating torii with elements of $A$.
That is, an argument analogous to the proof of our Baby Product theorem
based on the above dichotomy
completes the proof of the Product theorem.

\section{On the proof of the Polynomial Inverse theorem}

A large part of the proof (already contained in 
\cite{PyberSzabo:2010:growth})
is to show that if $S$ is a symmetric subset of $GL(n,\Fclosed)$
such that $\big|S^3\big|\le K|S|$,
then $S$ is contained in the union of polynomially many cosets of a
soluble-by-finite subgroup.

For the proof of this we first establish a generalisation of the
\fref{thm:LP-inequality}, called the ``Spreading theorem''
\cite[Theorem 49.]{PyberSzabo:2010:growth}. 
Roughly speaking it says the following.

For a subvariety $V$ of positive dimension we define the \emph{concentration}
of a finite set $S$ in $V$ as $\frac{\log|S\cap V|}{\dim(V)}$.

Let $S$ be a finite subset
in a connected linear algebraic group $G$ such that
$\CC_G(S)$ is finite.
If $G$ has a subvariety $X$ in which $S$ has much larger concentration
than in $G$ then we can find a connected closed subgroup $H\le G$
normalised by $S$
in which a small power of
$S$ has similarly large concentration.
(When $G$ is the simple algebraic group used to define a 
finite group of Lie type $L$ and $S$ generates $L$
then $H$ turns out to be $G$ itself.)

Next we have to introduce certain ``generalised torii''.
If $G$ is a simple algebraic group then a maximal torus $T$
can be obtained as the
centraliser of a regular semisimple element of $T$.
Moreover, most elements of $T$ are regular semisimple.
Since $T$ is abelian, it coincides with its centraliser $\CC_G(T)$.

In an arbitrary non-nilpotent linear algebraic group $G$
we consider a class of subgroups
we call CCC-subgroups with similar properties.
Our CCC-subgroups are the connected centralisers of connected subgroups
(which explains the name).

For a subvariety $X$ denote by $X^\gen$ the set of all $\dim(G)$-tuples 
$\ug\in\prod^{\dim(G)}X$
such that
the connected centralisers $\CC_G(\ug)^0$ and $\CC_G(X)^0$ are equal.

It turns out that if $X$ is a CCC-subgroup then $X^\gen$ is a
dense open subset of 
$\prod^{\dim(G)}X$. Moreover, if $Y$ is another CCC-subgroup,
then $X^\gen\cap Y^\gen=\emptyset$.
Finally we have $\CC_G\big(\CC_G(X)^\circ\big)^\circ=X$.

If we replace the torii in the proof of the Product theorem
with CCC-subgroups,
the proof goes through,
and shows that a small power of $S$ has very large concentration
in some subgroup $H$ of $G$ normalised by $S$.

An induction argument then shows
that $S$ is contained in polynomially many cosets of some soluble-by-finite
subgroup, as we wanted.

With various additional arguments this result can be upgraded as follows
\cite{PyberSzabo:2010:growth}.

\begin{thm} \label{thm:Hrushowsky-type}
  Let $1\in S$ be a finite symmetric subset of $GL(n,\BF)$
  such that $\big|S^3\big|\le K|S|$ for some $K\ge1$.
  Then $S$ can be covered by polynomially many cosets of some
  soluble-by-finite subgroup $\Gamma$ normalised by $S$.
  Moreover, $\Gamma$ has a soluble normal subgroup $H$
  such that $\Gamma\subset S^6H$.
\end{thm}

This is essentially a polynomial version of a result of Hrushovski
\cite{Hrushovski:approximate-subgroups:2012}.

Since a soluble-by-finite subgroup of $GL(n,\BC)$ has a soluble subgroup of
bounded index,
\fref{thm:Hrushowsky-type}
implies the Polynomial Inverse theorem for fields of characteristic zero.
As mentioned before, the characteristic zero case was first obtained by
Breuillard, Green and Tao
\cite{BrGrTao:ProductTheorem:2011}.

In the rest of this section $\BF$ denotes a field of characteristic $p>0$.

To complete the proof of \fref{thm:PSz-Polynomial-Inverse-thm}
next we have to settle the case of finite linear groups.

\begin{defn}
  As usual, $O_p(G)$ denotes the maximal
  normal $p$-subgroup of a finite group $G$.
  A group is called \emph{perfect}, if it has no abelian quotients.
  A finite group is called
  \emph{quasi-simple} if it is a perfect central extension
  of a finite simple group.
  We denote by $Lie^{*}(p)$ the set of central products of quasi-simple
  groups of Lie type of characteristic $p$.
\end{defn}

The following deep result is essentially due to Weisfeiler
\cite{Weisfeiler:post-classification-Jordan-thm:1984}.

\begin{prop} \label{prop:Weisfeiler}
  Let $G$ be a finite subgroup of $GL(n,\BF)$.
  Then $G$ has a normal
  subgroup $H$ of index at most $f(n)$ such that $H\ge O_p(G)$
  and $H/O_p(G)$ is the central product of an abelian $p'$-group and
  quasi-simple groups of Lie type of characteristic $p$, where the bound
  $f(n)$ depends on $n$.
\end{prop}

It was proved by Collins~\cite{Collins:modular-Jordan-thm:2008} that for $n\ge71$ one can take
$f(n)=(n+2)!$. Remarkably a (non-effective) version of the above result was
obtained by Larsen and Pink
\cite{Larsen-Pink:finite-subgrouls-of-alg-groups:2011}
without relying on the classification of finite simple groups.

It is an easy consequence of Weisfeiler's theorem,
that if $\degC(G)$ is ``large'' then
in fact $G/O_p(G)$ is in $Lie^{*}(p)$.
For such groups one can use \fref{lem:conjugation-trick} inductively
to prove, that if a symmetric set $S$ with $\big|S^3\big|\le K|S|$ 
projects onto $G/O_p(G)$,
then in fact $S^3=G$.

This result, and the Product theorem
are the main ingredients in the proof of the Polynomial Inverse theorem
for finite linear groups.

The following classical theorem of Malcev
(see \cite{Wehrfritz:Infinite-linear-groups:1973}) 
makes it possible to use finite group theory to study properties of finitely
generated linear groups.

\begin{prop}
  \label{prop:Malcev}
  Let $\Gamma$ be a finitely generated subgroup of $GL(n,\BF)$.
  For every finite set of elements $g_1,\dots,g_t$ of $\Gamma$
  there exists a finite field $\BK$ of the same characteristic
  and a homomorphism $\phi:\Gamma\to GL(n,\BK)$
  such that $\phi(g_1),\dots,\phi(g_t)$ are all distinct.
\end{prop}

In the final part of the proof of the Polynomial Inverse theorem
we have to further upgrade \fref{thm:Hrushowsky-type}.
Due to some lucky coincidences this can be done by combining the finite case
with Malcev's theorem.

\bibliographystyle{amsplain}
\bibliography{Bibliography}

\end{document}